\documentclass[11pt]{article}

\usepackage{amsmath}
\usepackage{amssymb}

 \newcommand{\beq}{\begin{equation}}
\newcommand{\eeq}{\end{equation}}

\newtheorem{theorem+}           {Theorem}      
\newtheorem{definition+} {Definition}
\newtheorem{lemma+}  {Lemma}
\newtheorem{corollary+}  {Corollary}
\newtheorem{proposition+} {Proposition}
\newtheorem{example+} {Example}

\newenvironment{theorem}{\begin{theorem+}\sl}{\end{theorem+}\rm}

\newenvironment{corollary}{\begin{corollary+}\sl}{\end{corollary+}\rm}

\newenvironment{proof}{\medbreak\noindent{\it Proof.}\rm}{\hfill$\square$\rm}

\renewcommand{\Bbb}{\mathbb}
\newcommand{\Z}{{\Bbb  Z}}

\newcommand{\Rn}{{ \Bbb R}^n}

\newcommand{\C}{{\Bbb  C}}
\newcommand{\Cn}{{\Bbb  C\sp n}}

\newcommand{\B}{{\Bbb  B}}

\newcommand{\I}{{\mathcal I}}
\renewcommand{\O}{{\mathcal O}}

\newcommand{\PSH}{{\operatorname{PSH}}}
\newcommand{\codim}{{\operatorname{codim}}}

\newcommand{\vph}{\varphi}
\newcommand{\vphy}{\varphi_y}

\begin{document}

\begin{center}
{\huge\bf Analyticity and propagation of plurisubharmonic
singularities}
\end{center}

\medskip
\begin{center}
{\Large\bf Alexander Rashkovskii}
\end{center}

\begin{abstract} A variant of Siu's analyticity theorem is proved for relative types
of plurisubharmonic functions. Some results on propagation of
plurisubharmonic singularities and maximality of pluricomplex Green
functions with analytic singularities are derived.
\end{abstract}

\section{Introduction}
Given a complex manifold $X$, let $\PSH(X)$ denote the class of all
plurisubharmonic functions on $X$ and $\PSH^-(X)$ its subclass of
all non-positive functions.

We will say that $u\in\PSH(X)$ has singularity at a point $\zeta\in
X$ if $u(\zeta)=-\infty$. A basic characteristic of the singularity
is its {\it Lelong number}
$$ \nu(u,\zeta)=
\liminf_{x\to \zeta}\frac{u(x)}{\log|\varsigma(x)|}=dd^cu\wedge
(dd^c\log|\varsigma(x)|)^{n-1}(\{\zeta\});$$ here $d=\partial +
\bar\partial$, $d^c= (\partial -\bar\partial)/2\pi i$, $n=\dim X$,
and $\varsigma$ are local coordinates on a neighbourhood of $\zeta$
with $\varsigma(\zeta)=0$.

A classical result due to Siu states that the the function $x\mapsto
\nu(u,x)$ is upper semicontinuous in the analytic Zariski topology;
this means that the set
$$S_c(u,X)=\{\zeta\in X:\nu(u,\zeta)\ge c\}$$
is an analytic variety of $X$ for any $u\in\PSH(X)$ and $c>0$. As a
consequence, for an arbitrary analytic variety $Z$, the value
$\nu(u,\zeta)$ is generically constant on $Z$, equal to
$\inf\{\nu(u,\zeta):\zeta\in Z\}$; it can be greater only on a
proper analytic subset of $Z$.

Siu's theorem was extended to directional Lelong numbers
$\nu(u,\zeta,a)$, $a\in\Rn_+$, by Kiselman \cite{Kis2}, and to
generalized (weighted) Lelong numbers $\nu(u,\vph)$ with respect to
exponentially H\"older continuous plurisubharmonic weights $\vph$ by
Demailly \cite{D1}. The analyticity theorems with respect to the
standard and directional Lelong numbers give important information
on asymptotic behaviour of plurisubharmonic functions near the
singularity points: for example, $u(x)\le c\log|\varsigma(x)|+O(1)$
as $x\to\zeta\in S_c(u,X)$. Relations between the weighted Lelong
numbers $\nu(u,\vph)$ and the asymptotic behaviour of $u$ are not
that direct.

In \cite{R7}, a notion of {\it relative type} $\sigma(u,\vph)$ of
$u$ with respect to a maximal plurisubharmonic weight $\vph$ was
introduced (see Section~2) and an analyticity theorem for the sets
$\{\zeta:\:\sigma(u,\vph_\zeta)\ge c\}=\{\zeta:\: u(x)\le
c\,\vph(x,\zeta) +O(1),\ x\to\zeta\}$ was proved, where
$\vph_\zeta(x)=\vph(x,\zeta)\in\PSH(X\times X)$ is such that
$\vph_\zeta^{-1}(-\infty)=\zeta$, $(dd^c\vph)^n=0$ on $\{x\neq
\zeta\}$, and $e^\vph$ is H\"older continuous with respect to
$\zeta$. The extra condition (comparing to Demailly's result) on
$(dd^c\vph)^n$ is quite essential. Take, for example, the function
$\vph(x,\zeta)=\max\{\log|x_1-\zeta_1|+\log(|x_1-\zeta_1|+|x_2|),
\log|x_2-\zeta_2|\}$ in $\C^2\times\C^2$; one has $\log|x_1|\le
\vph(x,\zeta)+O(1)$ precisely when
$\zeta\in\{(0,\zeta_2):\zeta_2\neq 0\}$ that is not an analytic
variety. The reason here is that the values of the weighted Lelong
numbers $\nu(u,\vph_\zeta)$ and relative types
$\sigma(u,\vph_\zeta)$ depend on the singularity of $\vph$ in
opposite ways: while any jump of the singularity of $\vph$ at a
particular point $\zeta$ just {\sl increases} the value of
$\nu(u,\vph_\zeta)$, it {\sl diminishes} the type
$\sigma(u,\vph_\zeta)$.

Here we present a more general analyticity result
(Theorem~\ref{theo:typean}) for the relative types. Its main feature
is that we allow the singularity sets $\vph_\zeta^{-1}(-\infty)$
consisting of several points, which makes it possible to apply the
result to weights generated by finite holomorphic mappings. Another
benefit is that the analyticity concerns a parameter space (as in
\cite[Th\'eor\`eme 4.14]{D1}), which can thus give additional
information on the asymptotic behaviour even at a fixed point (see,
for example, Corollary~\ref{cor:var}). We derive some results on
propagation of plurisubharmonic singularities
(Corollary~\ref{cor:finite} and Theorem~\ref{theo:integr}), which in
turn imply certain global maximality properties of pluricomplex
Green functions with non-isolated analytic singularities
(Corollary~\ref{cor:gr}).

\section{Preliminaries}
Throughout the note, the following notions will be used.

A function $u\in\PSH(X)$ is said to be {\it maximal} on an open set
$U\Subset X$ if for any $v\in\PSH(X)$ the condition $v\le u$ on
$X\setminus U$ implies $v\le u$ on the whole $X$. A locally bounded
$u$ is maximal on $U$ if and only if $(dd^cu)^n=0$ there, $n=\dim
X$.

Given a Stein manifold $X$, let us have a finite set ${\mathcal
Z}=\{\zeta_1,\ldots,\zeta_k\}\subset X$ and functions
$\varphi_1,\ldots,\varphi_k$ such that $\varphi_j$ is
plurisubharmonic near $\zeta_j$, locally bounded  and maximal on a
punctured neighbourhood of $\zeta_j$, and
$\varphi_j(\zeta_j)=-\infty$. The function
$$G_{A,\{\varphi_j\}}(z)=\sup\,\{u(z):\: u\in PSH^-(X),\ u\le \varphi_j\
{\rm near\ }\zeta_j,\ 1\le j\le k\}$$ is the {\it Green--Zahariuta
function} of $X$ with the singularity $\varphi=\{\varphi_j\}$. The
notion was introduced, for the continuous weights $\varphi_j$, in
\cite{Za0}, see also \cite{Za}; the general case was treated in
\cite{R7}. The function $G_{A,\varphi}$ is plurisubharmonic in $X$,
maximal on $X\setminus{\mathcal Z}$ and satisfies
$G_{A,\varphi}(x)=\varphi_j(x)+O(1)$ as $x\to\zeta_j$.

Let $\vph\in PSH(X)$ be locally bounded on $X\setminus{\mathcal Z}$
and such that its restriction to a neighbourhood of each point
$\zeta_j$ is a maximal weight equivalent to $\varphi_j$ in the sense
$\lim\vph_j(x)/\vph(x)=1$; for example, one can take
$\vph=G_{A,\{\varphi_j\}}$. The {\it relative type} $\sigma(u,\vph)$
of $u$ with respect to $\vph$ was introduced in \cite{R7} as
$$\sigma(u,\vph)=\liminf_{\vph(x)\to-\infty}\frac{u(x)}{\vph(x)}.$$
In other words,
$$\sigma(u,\vph_y)=
\lim_{r\to-\infty}r^{-1}\Lambda(u,\vph,r),$$ where
$\Lambda(u,\vph,r):=\sup\{u(x):\: \vph(x)<r\}$.

\section{Analyticity theorem}
Let now $X$ be a Stein manifold of dimension $n$ and $Y$ be a
complex manifold of dimension $m$. Let $R:Y\to (-\infty,\infty]$ be
a lower semicontinuous function on $Y$. We consider a continuous
plurisubharmonic function $\vph:X\times Y\to [-\infty,\infty)$ such
that:
\begin{enumerate} \item[(i)] $\vph(x,y)<R(y)$ on $X\times Y$;
\item[(ii)] the set ${\mathcal Z}(y)=\{x:\vph(x,y)=-\infty\}$ is
finite for every $y\in Y$; \item[(iii)] for any $y_0\in Y$ and
$r<R(y_0)$ there exists a neighbourhood $U$ of $y_0$ such that the
set $\{(x,y):\: \vph(x,y)<r,\: y\in U\}\Subset X\times Y$;
\item[(iv)] $(dd^c\vph)^n=0$ on $\{ \vph(x,y)>-\infty\}$;
\item[(v)] $e^{\vph(x,y)}$ is locally H\"older continuous in $y$:
every point $(x_0,y_0)\in X\times Y$ has a neighbourhood $\omega$
such that \begin{equation}\label{eq:holder}
|e^{\vph(x,y_1)}-e^{\vph(x,y_2)}|\le
M|\varsigma(y_1)-\varsigma(y_2)|^\beta, \quad (x,y_j)\in\omega,
\end{equation} for some $M,\,\beta>0$ and suitable coordinates
$\varsigma$ on $Y$.
\end{enumerate}

The function $\vph_y(x)=\vph(x,y)$ is a maximal plurisubharmonic
weight with poles at ${\mathcal Z}(y)$; we will write this as
$\vph_y(x)\in MW_{{\mathcal Z}(y)}$. In particular, given $u\in
PSH(X)$, the function
$$r\mapsto \Lambda(u,\vph_y,r):=\sup\{u(x):\: \vph_y(x)<r\}$$
is convex and there exists the limit
$$\sigma(u,\vph_y)=
\lim_{r\to-\infty}r^{-1}\Lambda(u,\vph_y,r) =\liminf_{x\to{\mathcal
Z}(y)}\frac{u(x)}{\vph(x,y)},$$ the relative type of $u$ with
respect to the weight $\vph_y$. We have thus
\begin{equation}\label{eq:mainb} u(x)\le
\sigma(u,\vph_y)\vph(x,y)+O(1),\quad x\to {\mathcal
Z}(y).\end{equation}

Denote $$S_c(u,\vph,Y)=\{y\in Y:\: u(x)\le c\vph(x,y) +O(1)\ {\rm
as\ }x\to{\mathcal Z}(y)\}.$$ Equivalently, $S_c(u,\vph,Y)=\{y\in
Y:\: \sigma(u,\vphy)\ge c\}$.

\begin{theorem}\label{theo:typean}
Let a continuous function $\vph\in PSH(X\times Y)$ satisfy the above
conditions (i)--(v). Then for every $u\in PSH(X)$ and $c>0$, the set
$S_c(u,\vph,Y)$ is an analytic variety.
\end{theorem}

\begin{proof} We will follow the lines of the proof of
\cite[Theorem~7.1]{R7}, which in turn is an adaptation of Kiselman's
and Demailly's proofs of the corresponding variants of Siu's
theorem. Note that although the proof is quite short, it is based on
such deep results as Demailly's theorem on plurisubharmonicity of
the function $\Lambda(u,\vph_y,r)$ and the Bombieri--H\"ormander
theorem.

By \cite[Theorem~6.11]{D3}, the function $\Lambda(u,\vphy,{\rm
Re}\,\xi)$ is plurisubharmonic on the set $\{(y,\xi)\in Y\times\C:\:
{\rm Re}\,\xi<R(y)\}$. Fix a pseudoconvex domain $D\Subset Y$ and
denote $R_0=\inf\,\{R(y):\: y\in D\}>-\infty$. Given $a>0$, the
function $$(u,\xi)\mapsto\Lambda(u,\vphy,{\rm Re}\,\xi)-a\,{\rm
Re}\,\xi$$ is thus plurisubharmonic in $D\times\{{\rm
Re}\,\xi<R_0\}$ and independent of ${\rm Im}\,\xi$, so by Kiselman's
minimum principle \cite{Kis8}, the function $$
U_a(y)=\inf\{\Lambda(u,\vphy,r)-a(r-R_0):\:r<R_0\}
$$ is plurisubharmonic in $D$.

Let $y\in D$. If $a>\sigma(u,\vphy)$, then $\Lambda(u,\vphy,r)>
a(r-R_0)$ for all $r\le r_0<R_0$. If $r_0<r<R_0$, then
$\Lambda(u,\vphy,r)- a(r-R_0)> \Lambda(u,\vph_y,r_0)$. Therefore
$U_a(y)>-\infty$.

Now let $a<\sigma(u,\vphy)$. In view of property (iii) and estimate
(\ref{eq:mainb}), the exponential H\"older continuity
(\ref{eq:holder}) implies the bound
$$\Lambda(u,\vph_z,r)\le
\Lambda(u,\vphy,\log(e^r+M|\varsigma(z)|^\beta))\le
\sigma(u,\vphy)\log(e^r+M|\varsigma(z)|^\beta)+C$$ in a
neighbourhood $U_y$ of $y$ with the coordinates $\varsigma$ chosen
so that $\varsigma(y)=0$. Denote $r_z=\beta\log|\varsigma(z)|$, then
\begin{equation}\label{eq:holb1} U_a(z)\le \Lambda(u,\vph_z,r_z)- ar_z
\le (\sigma(u,\vphy)-a)\beta\log|\varsigma(z)|+C_1, \quad z\in
U_y.\end{equation}

Given $a,b>0$, let $Z_{a,b}$  be the set of points $y\in D$ such
that the function $\exp(-b^{-1}U_a)$ is not integrable near $y$. As
follows from the H\"ormander--Bombieri--Skoda theorem
\cite[Theorem~4.4.4]{Ho}, all the sets $Z_{a,b}$ are analytic.

If $y\not\in S_c(u,\vph,D)$ and $\sigma(u,\vphy)<a<c$, then
$U_a(y)>-\infty$ and so, by Skoda's theorem
\cite[Theorem~4.4.5]{Ho}, $y\not\in Z_{a,b}$ for all  $b>0$.

If $y\in S_c(u,\vph,D)$, $a<c$, and $b< (c-a)\beta (2m)^{-1}$, then
(\ref{eq:holb1}) implies $y\in Z_{a,b}$. Thus, $S_c(u,\vph,D)$
coincides with the intersection of all the sets $Z_{a,b}$ with $a<c$
and $b<(c-a)\beta (2m)^{-1}$, and is therefore analytic.
\end{proof}

\section{Dependence on coordinates}
By a classical result (again due to Siu), standard Lelong numbers
are independent of the choice of coordinates. The following
statement can be viewed as a bridge between Siu's analyticity and
invariance theorems.

\begin{corollary}\label{cor:var} Let $\vph\in MW_0$ satisfy $|e^{\vph(a)}
-e^{\vph(b)}|\le M|a-b|^\beta$, $\beta>0$, on a pseudoconvex
neighbourhood $X$ of $0\in \Cn$, and let $Y$ be a complex manifold
in $GL_n(\C)$. Then for every $u\in PSH(X)$, the sets
$$\{(\zeta,A)\in X\times Y: u(x)\le \vph(Ax-\zeta)+O(1)\ {\rm as\
}x\to A^{-1}\zeta\}$$ and
$$\{(\zeta,A)\in X\times Y: u(x)\le \vph(A(x-\zeta))+O(1)\ {\rm as\
}x\to \zeta\}$$ are analytic varieties in $X\times Y$. In
particular, the set $$S(u,\vph,Y)=\{A\in Y: u(x)\le \vph(Ax)+O(1)\
{\rm as\ } x\to 0\}$$ is analytic in $Y$. The functional
$u\mapsto\sigma(u,\vph_A)$, where $\vph_A(x)=\vph(Ax)$, is
independent of $A\in GL_n(\C)$ if and only if
$\vph(x)=c\log|x|+O(1)$ for some constant $c>0$.
\end{corollary}

\begin{proof} The analyticity follows directly from
Theorem~\ref{theo:typean}. To prove the last assertion, consider the
Green--Zahariuta function $G_\vph$ for the singularity $\vph$ in the
unit ball $\mathbb B$. Since $\vph(x)=\vph(Ax)+O(1)$ for any unitary
$A$, we have $G_\vph(x)=\chi(\log|x|)$, where $\chi$ is a convex
increasing function on $(-\infty,0)$. The equation
$(dd^cG_\vph)^n=0$ outside $0$ implies $\chi''=0$, and the condition
$G_\vph=0$ on $\partial \mathbb B$ gives then $\chi(t)=c\,t$, $c>0$.
\end{proof}
\medskip

{\it Remark}. For the case $\vph(x)=\max_k \log|x_k|^{a_k}$ and
$Y=GL_n(\C)$, similar analyticity theorems were proved in \cite{D}
and \cite{Kis3}.

\section{Analytic singularities}
Let $F:X\times Y\to \C^{n}$ be a holomorphic mapping such that its
zero set $|Z_F|$ is of codimension $n$ and moreover,
$|Z_F|\cap\{(x,y_0):x\in X\}$ is finite for any $y_0\in Y$. Then the
function $\vph(x,y)=\log|F(x,y)|$ satisfies conditions (i)--(v) on
$X'\times Y$ for any domain $X'\Subset X$; condition (iv) follows
from King's formula $(dd^c\log|F|)^{n}=[Z_F]$. This observation can
be used in finding analytic majorants for plurisubharmonic
singularities.

\begin{corollary}\label{cor:finite} Let $f=(f',f'')$ be a finite equidimensional
holomorphic mapping on a complex manifold $X$. If $u\in \PSH(X)$
satisfies $u\le\log|f'|+O(1)$ on an open set $\omega\subset X$
intersecting every irreducible component of the zero set of $f'$,
then $u\le\log|f'|+O(1)$ locally on $X$.
\end{corollary}

\begin{proof} Let $\vph_N(x,y)=\log(|f'(x)-f'(y)|+
|f''(x)-f''(y)|^N)$, $N\in\Z_+$, and let $X'\Subset X$ be such that
$\omega'=X'\cap\omega$ intersects all irreducible components of the
set $Z'=\{x\in X': f'(x)=0\}$. Then, by Theorem~\ref{theo:typean},
$S(u,\vph,X')$ is an analytic variety. By the assumption,
$S(u,\vph,X')\cap\omega'\supset S(\log|f'|,\vph,X')\cap\omega'$.
Therefore, $S(u,\vph,X')$ contains all irreducible components of
$S(\log|f'|,\vph,X')$ that pass through $\omega$. Observe now that
$S(\log|f'|,\vph,X')= Z'$, which implies $u\le \vph_N+C$ on $Z'$.

Given $a\in Z'$, we can assume $D=\{x:\max\{|f'(x)|,
|f''(x)-f''(a)|\}<1\}\Subset X'$. Therefore, $u\le g_N+C$, where
$g_N(x)=\max\{\log|f'(x)|, N\log|f''(x)-f''(a)|\}$ is the
Green--Zahariuta function for the singularity $\vph_N$ in $D$.
Taking $N\to\infty$, we get $u\le\log|f'|+C$ in $D$.
\end{proof}
\medskip

A more accurate analysis allows us to weaken the assumptions on the
mapping $f'$ in Corollary~\ref{cor:finite}. To this end, it is
convenient to use the notion of complex spaces.

For a closed complex subspace $A$ of $X$, let
${\I}_A=({\I}_{A,x})_{x\in X}$ be the associated coherent sheaf of
ideals in the sheaf $\O_X$ of germs of holomorphic functions on $X$,
and let $|A|$ be the variety in $X$ locally defined as the common
set of zeros of holomorphic functions with germs in ${\I}_A$, i.e.,
$|A|=\{x:{\I}_{A,x}\neq \O_{X,x}\}$.

Recall that an ideal ${\mathcal J}\subset \I\subset \O_{X,x}$ is
called a {\it reduction} of $\I$ if its integral closure coincides
with that of $\I$; the {\it analytic spread} of $\I$ equals the
minimal number of generators of its reductions \cite{NR}.

We will say that a complex space $A$ is {\it integrally generic at}
$x\in |A|$ if the analytic spread of ${\I}_{A,x}$ equals
$\codim_x|A|$. This is equivalent to saying that there exist
functions $h_k\in{\I}_{A,x}$, $ k=1,\ldots, \codim_x|A|$, such that
$\log|h|=\log|f|+O(1)$, where $f=(f_1,\ldots,f_{s})$ are generators
of ${\I}_{A,x}$, see \cite{NR}. A space $A$ is {\it integrally
generic} if it is so at each $x\in|A|$.

We will write $u\le\log|\I_A|$ if a function $u$ satisfies
$u\le\log|f|+O(1)$ for local generators $f$ of $\I_A$.

\begin{theorem}\label{theo:integr} Let $A$ be an integrally generic complex
space on $X$ and $\omega$ be an open set intersecting every
irreducible component of $|A|$. If a function $u\in\PSH(X)$
satisfies $u\le \log|\I_A|$ on $\omega$, then it satisfies the
relation everywhere in $X$.
\end{theorem}

\begin{proof} Denote by $Z_l$, $l=1,2,\ldots$, the irreducible
components of $|A|$. We will first prove near all points of the set
$ Z_l^*=Z_l\setminus \cup_{k\neq l}Z_k.$

Let $\codim\, Z_l=p$. For an arbitrary point $z\in
Z_l^*\cap\partial\omega$, there is a neighbourhood $U$ of $z$, a
holomorphic mapping $h:U\to\C^p$, and a linear mapping
$U\to\C^{n-p}$, such that $|A|\cap U=Z_l^*\cap U$, $\log|h|\le
\log|\I_A|$, and for every $y\in V$, the mapping
$F_y:x\mapsto(h(x)-h(y),L(x)-L(y))$ is finite in $U$. By
Corollary~\ref{cor:finite}, we get then $u\le
\log|h|+O(1)\le\log|\I_A|$ on $\omega\cup U$.

Now we can repeat the procedure with $\omega\cup U$ instead of
$\omega$. Since the sets $Z_l^*$ are connected, it gives us the
desired bounded near every point of $|A|^*=\cup_l Z_l^*$.

The rest points can be treated as in the proof of
\cite[Lemma~4.2]{RSig2}. Namely, fix a point $z\in |A|\setminus
|A|^*$, $\codim_z|A|=p$. By Thie's theorem, there exist local
coordinates $x=(x',x'')$, $x'=(x_1,\ldots,x_p)$,
$x''=(x_{p+1},\ldots,x_n)$, centered at $z$, and balls $\B'\subset
{\C}^p$, $\B''\subset {\C}^{n-p}$ such that $\B'\times \B''\Subset
V$, $|A|\cap(\B'\times \B'')$ is contained in the cone $\{|x'|\le
\gamma |x''|\}$ with some constant $\gamma>0$, and the projection of
$|A|\cap(\B'\times \B'')$ onto $\B''$ is a ramified covering with a
finite number of sheets.

Let $h=(h_1,\ldots,h_p)$ satisfy $\log|h|\le\log|\I_A|$ on $V$. Take
$r_1=2\gamma r_2$ with a sufficiently small $r_2>0$ so that
$\B_{r_1}'\subset \B'$ and $\B_{r_2}''\subset \B''$, then for some
$\delta>0$ we have $|h|\ge\delta$ on $\partial \B_{r_1}'\times
\B_{r_2}''$.

Given a point $x_0''\subset \B_{r_2}''$, denote by $R(x_0'')$ and
$S(x_0'')$ the intersections of the set $\B_{r_1}'\times\{x_0''\}$
with the varieties $|A|$ and $|A|\setminus |A|^*$, respectively.
Since the projection is a ramified covering, $R(x_0'')$ is finite
for any $x_0''\in \B_{r_2}''$, while $S(x_0'')$ is empty for almost
all $x_0''\in \B_{r_2}''$ because $\dim S\le n-p-1$; we denote the
set of all such generic $x_0''$ by $E$.

Given  $x_0''\in E$, the function $v(x')=\log(|h(x',x_0'')|/\delta)$
is nonnegative on $\partial \B_{r_1}'$ and maximal on
$\B_{r_1}'\setminus R(x_0'')$, since the map
$h(\cdot,x_0''):\B_{r_1}'\to{\C}^p$ has no zeros outside $R(x_0'')$.
Since $u$ satisfies $u\le \log|h| + O(1)$ locally near points of
$|A|^*$, we have then $u(x',x_0'')\le v(x')+C$ on the whole ball
$\B_{r_1}'$, where $C=\sup_V u$.

As $x_0''\in E$ is arbitrary, this gives us $u\le \log|h|-\log\delta
+C$ on $\B_{r_1}'\times E$. The continuity of the function $\log|h|$
extends this relation to the whole set $\B_{r_1}'\times \B_{r_2}''$,
which completes the proof.
\end{proof}

\section{Green functions}
The result can be applied to investigation of maximality properties
for Green functions with analytic singularities.

The {\it Green function $G_A$ with singularities along} a complex
space $A$ is the upper envelope of the class of all functions $u\in
PSH^-(X)$ such that $u\leq \log|\I_A|$. This function is
plurisubharmonic in $X$ and satisfies $G_A\leq \log|\I_A|$, see
\cite{RSig2}.

When $|A|$ is discrete, $G_A$ is maximal on $X\setminus |A|$. In the
case $\dim|A|>0$, the Green function has additional maximality
properties. Namely, if $\I_A$ has $p<n$ global generators, then
$G_A$ is maximal on the whole $X$, and for an arbitrary complex
space $A$, the function $G_A$ is {\it locally maximal} outside a
discrete subset $J_A$ of $|A|$ consisting of all points $x\in|A|$
such that the analytic spread of ${\I}_{A,x}$ equals $n$
\cite[Theorem~4.3]{RSig2}; in \cite{R}, $J_A$ was called the {\it
complete indeterminacy locus}. (A function $v$ is said to be locally
maximal on an open set $\omega$ if every point of $\omega$ has a
neighbourhood where $v$ is maximal.)

We do not know if the function $G_A$ is always maximal on
$X\setminus J_A$; what we can prove is the following result.

\begin{corollary}\label{cor:gr} If $A$ is an arbitrary closed
complex space on $X$, then the function $G_A$ is maximal outside an
analytic subset $J$ of $|A|$, nowhere dense in each
positive-dimensional component of $|A|$. If $\dim X=2$, then $J$
coincides with the complete indeterminacy locus $J_A$. If $A$ is
integrally generic, then $J=\emptyset$.
\end{corollary}

\begin{proof} By \cite[Proposition~3.5]{RSig2}, the set $|A|$ can
be decomposed into the disjoint union of local (not necessarily
closed) analytic varieties $J^k$, $1\le k\le n$, such that
$\codim\,J^k\ge k$ and for each $a\in J^k$, the ideal $\I_{A,a}$ has
analytic spread at most $k$. In view of Theorem~\ref{theo:integr},
this implies the claims.
\end{proof}

\bibliographystyle{amsalpha}

\vskip1cm

Tek/Nat, University of Stavanger, 4036 Stavanger, Norway

\vskip0.1cm

{\sc E-mail}: alexander.rashkovskii@uis.no

\end{document}